\documentclass{amsart}

\usepackage{amssymb}
\usepackage{amsthm}
\usepackage{amsmath}
\usepackage{amsfonts}
\usepackage[mathscr]{eucal}
\usepackage{enumerate}
\usepackage{hyperref}

\newtheorem{thm}{Theorem}
\newtheorem{lem}[thm]{Lemma}
\newtheorem{cor}[thm]{Corollary}

\begin{document}

\title[Derivations and linear functions along rational functions]{Derivations and linear functions\\ along rational functions}

\author[E.~Gselmann]{Eszter Gselmann}

\address{Institute of Mathematics\\
University of Debrecen\\
P. O. Box: 12.\\
Debrecen\\
H--4010\\
Hungary}

\email{gselmann@science.unideb.hu}

\thanks{This research has been supported by the Hungarian Scientific Research Fund (OTKA)
Grant NK 814 02 and by the T\'{A}MOP 4.2.1./B-09/1/KONV-2010-0007 project implemented
through the New Hungary Development Plan co-financed by the European Social Fund and
the European Regional Development Fund.}



\date{\today}

\begin{abstract}
The main purpose of this paper is to give characterization theorems on derivations as well as on linear functions. 
Among others the following problem will be investigated:
Let $n\in\mathbb{Z}$, $f, g\colon\mathbb{R}\to\mathbb{R}$ be additive functions, 
$\left(\begin{array}{cc}
a&b\\
c&d
\end{array}
\right)\in\mathbf{GL}_{2}(\mathbb{Q})$ be arbitrarily fixed, 
and let us assume that the mapping 
\[
 \phi(x)=g\left(\frac{ax^{n}+b}{cx^{n}+d}\right)-\frac{x^{n-1}f(x)}{(cx^{n}+d)^{2}} 
\quad 
\left(x\in\mathbb{R}, cx^{n}+d\neq 0\right)
\]
satisfies some regularity on its domain (e.g. (locally) boundedness, continuity, measurability). 
Is it true that in this case the above functions can be represented as a sum of a derivation and a linear function? 
Analogous statements ensuring linearity will also be presented. 
\end{abstract}

\maketitle

\section{Introduction and preliminaries} 

Throughout this paper $\mathbb{N}$ denotes the set of the positive integers, and $\mathbb{Z}, \mathbb{Q}$, and
$\mathbb{R}$ have the usual meaning. 

The aim of this work is to prove characterization theorems on derivations as well as on linear functions. 
Therefore, firstly we have to recall some definitions and auxiliary results. 

A function $f:\mathbb{R}\rightarrow\mathbb{R}$ is called an \emph{additive} function
if,
\[
f(x+y)=f(x)+f(y)
\]
holds for all $x, y\in\mathbb{R}$. 

We say that an additive
function $f:\mathbb{R}\rightarrow\mathbb{R}$ is a \emph{derivation} if
\[
f(xy)=xf(y)+yf(x)
\]
is fulfilled for all $x, y\in\mathbb{R}$.

The additive function $f\colon\mathbb{R}\to\mathbb{R}$ is termed to be a \emph{linear function} if 
$f$ is of the form 
\[
 f(x)=f(1)x 
\qquad 
\left(x\in\mathbb{R}\right). 
\]

It is easy to see from the above definition that every derivation
$f:\mathbb{R}\rightarrow\mathbb{R}$ satisfies equation
\[
\tag{$\ast$} f(x^{k})=kx^{k-1}f(x)
\quad
\left(x\in\mathbb{R}\setminus\left\{0\right\}\right)
\]
for arbitrarily fixed $k\in\mathbb{Z}\setminus\left\{0\right\}$.
Furthermore, the converse is also true, in the following sense:
if $k\in\mathbb{Z}\setminus\left\{0, 1\right\}$ is fixed and
an additive function $f:\mathbb{R}\rightarrow\mathbb{R}$ satisfies
$\left(\ast\right)$, then $f$ is a derivation, see e.g.,
Jurkat \cite{Jur65}, Kurepa \cite{Kur64}, and
Kannappan--Kurepa \cite{KK70}.

Concerning linear functions, Jurkat \cite{Jur65} and,
independently, Kurepa \cite{Kur64} proved that every additive function
$f:\mathbb{R}\rightarrow\mathbb{R}$ satisfying
\[
f\left(\frac{1}{x}\right)=\frac{1}{x^{2}}f(x)
\quad
\left(x\in\mathbb{R}\setminus\left\{0\right\}\right)
\]
has to be linear.

In \cite{NH68} A.~Nishiyama and S.~Horinouchi investigated
additive functions $f:\mathbb{R}\rightarrow\mathbb{R}$
satisfying the additional equation
\[
f(x^{n})=cx^{k}f(x^{m})
\quad
\left(x\in\mathbb{R}\setminus\left\{0\right\}\right),
\]
where $c\in\mathbb{R}$ and $n, m, k\in\mathbb{Z}$ are arbitrarily
fixed. 

Furthermore, the above problem was generalized by Pl.~Kannappan and S.~Kurepa 
in \cite{KK70} by proving the following statement.

\begin{thm}\label{T1}
Let $f, g:\mathbb{R}\rightarrow\mathbb{R}$ be additive functions and
$n, m\in\mathbb{Z}\setminus\left\{0\right\}$, $n\neq m$.
Suppose that
\[
f(x^{n})=x^{n-m}g(x^{m})
\]
holds for all $x\in\mathbb{R}\setminus\left\{0\right\}$.
Then the functions $F, G:\mathbb{R}\rightarrow\mathbb{R}$ defined by
\[
F(x)=f(x)-f(1)x
\quad
\text{and}
\quad
G(x)=g(x)-g(1)x
\quad
\left(x\in\mathbb{R}\right)
\]
are derivations and $nF(x)=mG(x)$ is fulfilled for all $x\in\mathbb{R}$.
\end{thm}

Here we remark that the 'approximate' variant of the above theorem was dealt with in \cite{BG10}. 

Furthermore, in a series of papers  (see \cite{HKR99, HKR00, HK00, HK00a, HKR01}) F.~Halter--Koch and 
L.~Reich proved several characterization theorems concerning derivations as well as field 
homomorphisms.

Let $\mathbb{K}$ be a field containing $\mathbb{Q}$, $n\in\mathbb{Z}\setminus\left\{0\right\}$, 
$\left(\begin{array}{cc}
a&b\\
c&d
\end{array}
\right)\in\mathbf{GL}_{2}(\mathbb{Q})$ and let 
$f, g\colon\mathbb{K}\to\mathbb{K}$ be additive function so that 
\[
\tag{$\star$} f\left(\frac{ax^{n}+b}{cx^{n}+d}\right)=\frac{x^{n-1}g(x)}{(cx^{n}+d)^{2}}, 
\]
respectively, 
\[
\tag{$\blacktriangle$} f\left(\frac{ax^{n}+b}{cx^{n}+d}\right)=\frac{ag(x)^{n}+b}{cg(x)^{n}+d}
\]
holds for all possible values of $x$. In \cite{HK00} it is proved that 
equation $\left(\star\right)$ (under a mild condition) implies for the function $g$ that the function 
$G\colon\mathbb{K}\to\mathbb{K}$ defined by 
\[
G(x)=g(x)-g(1)x 
\qquad 
\left(x\in\mathbb{K}\right)
\]
is a derivation. 
Furthermore, in \cite{HKR01} the authors succeed to prove that equation $\left(\blacktriangle\right)$ 
furnishes that the mapping $g(1)^{-1}\cdot g\colon\mathbb{K}\to\mathbb{K}$ is a field automorphism. 

In this work we will extend the above mentioned results, and (among others) the following problem will be dealt with. 
Let $n\in\mathbb{Z}$, $f, g\colon\mathbb{R}\to\mathbb{R}$ be additive functions, 
$\left(\begin{array}{cc}
a&b\\
c&d
\end{array}
\right)\in\mathbf{GL}_{2}(\mathbb{Q})$ be arbitrarily fixed, 
and let us assume that the mapping 
\[
 \phi(x)=g\left(\frac{ax^{n}+b}{cx^{n}+d}\right)-\frac{x^{n-1}f(x)}{(cx^{n}+d)^{2}} 
\quad 
\left(x\in\mathbb{R}, cx^{n}+d\neq 0\right)
\]
satisfies some regularity on its domain (e.g. (locally) boundedness, continuity, measurability). 
Is it true that in this case the above functions can be represented as a sum of a derivation and a linear function?

During the second section the concept of multi-additive functions will be used. Therefore, at this part we will list some definitions and 
statements that we will use subsequently. 

Let $G, H$ be abelian groups, let $h\in G$ be arbitrary and consider a function $f:G\rightarrow H$. 
The \emph{difference operator} $\Delta_{h}$ with the span $h$ of the function $f$ is defined by
\[
 \Delta_{h}f(x)=f(x+h)-f(x)
\qquad 
\left(x\in G\right). 
\]
The iterates $\Delta^{n}_{h}$ of $\Delta_{h}$, $n=0, 1, \ldots$ are defined by the recurrence 
\[
 \Delta^{0}_{h}f=f, \qquad 
\Delta^{n+1}_{h}f=\Delta_{h}\left(\Delta^{n}_{h}f\right)
\qquad 
\left(n=0, 1, \ldots\right)
\]
Furthermore, the superposition of several difference operators will be denoted shortly 
\[
 \Delta_{h_{1} \ldots h_{n}}f=\Delta_{h_{1}}\ldots \Delta_{h_{n}}f, 
\]
where $n\in\mathbb{N}$ and $h_{1}, \ldots, h_{n}\in G$. 

Let $n\in\mathbb{N}$ and $G, H$ be abelian groups.
A function $F:G^{n}\rightarrow H$ is
called \emph{$n$--additive} if, for every
$ i \in \{\, 1 \,,\, 2 \,,\, \dots \,,\, n \,\} $
and for every
$ x_{1}, \ldots, x_{n}, y_{i}\in G \,$,
\begin{multline*}
F\left(x_{1}, \ldots, x_{i-1}, x_{i} + y_{i}, x_{i+1}, \ldots, x_{n}\right)
\\
=F\left(x_{1}, \ldots, x_{i-1}, x_{i}, x_{i+1}, \ldots, x_{n}\right)
+F\left(x_{1}, \ldots, x_{i-1}, y_{i}, x_{i+1}, \ldots, x_{n}\right),
\end{multline*}
i.e., $F$ is additive in each of its variables
$x_{i}\in G$, $i=1, \ldots, n$. 
For the sake of brevity we use the notation $G^{0}=G$ and we call 
constant functions from $G$ to $H$ $0$--additive functions. 
Let $F:G^{n}\rightarrow H$ be an arbitrary function. 
By the \emph{diagonalization (or trace)} of $F$ we understand the
function
$f:G\rightarrow H$ arising from
$F$ by putting all the variables (from $G$)
equal:
\[
f(x)=F(x, \ldots, x)
\qquad
\left(x\in G\right).
\]
It can be proved by induction that for any symmetric, $n$--additive function
$F:G^{n}\rightarrow H$ the equality 
\[
 \Delta_{y_{1}, \ldots, y_{k}}f(x)=
\left\{
\begin{array}{lcl}
 n!F(y_{1}, \ldots, y_{n})& \text{for}& k=n \\
0 & \text{for} & k>n
\end{array}
\right.
\]
holds, whenever $x, y_{1}, \ldots, y_{n}\in G$, where $f:G\rightarrow H$ denotes the trace of the 
function $F$. This means that a symmetric, $n$--additive function is uniquely determined by its trace. 

The function $f:G\rightarrow H$ is called a \emph{polynomial function} of degree at most $n$, where 
$n$ is a nonnegative integer, if 
\[
 \Delta_{y_{1}, \ldots, y_{n+1}}f(x)=0
\]
is satisfied for all $x, y_{1}, \ldots, y_{n+1}\in G$. 

\begin{thm}[Kuczma \cite{Kuc09}, Sz\'{e}kelyhidi \cite{Sze91}]\label{T1.1}
The function $p:G\rightarrow H$ is a polynomial at degree at most $n$ if and only if there 
exist symmetric, $k$--additive functions $F_{k}:G^{k}\rightarrow H$, $k=0, 1, \ldots, n$ such that 
\[
 p(x)=\sum_{k=0}^{n}f_{k}(x) 
\qquad 
\left(x\in G\right), 
\]
where $f_{k}$ denotes the trace of the function $F_{k}$, $k=0, 1, \ldots, n$. 
Furthermore, this expression for the function $p$ is unique in the sense that the 
functions $F_{k}$, which are not identically zero,  are uniquely determined. 
\end{thm}

During the proof of our main theorem, the basic idea is to apply two statements from 
\cite{Sze85, Sze91} concerning polynomial functions. 
Despite the fact that the following results will be used only in case the topological group is 
$\mathbb{R}$ and the topological linear space is also $\mathbb{R}$, here we present the general case. 
The reason for it is, that the problem, we investigate, can be formulated in more general circumstances. 
Perhaps in the general case the two cited theorems of L.~Sz\'{e}kelyhidi would also play an important role. 

\begin{thm}[Sz\'{e}kelyhidi \cite{Sze85, Sze91}]\label{T1.2}
 Let $G$ be an abelian group and let $X$ be a locally convex topological 
linear space. If a polynomial $p:G\rightarrow X$ is bounded on $G$, then it is 
constant. 
\end{thm}

\begin{thm}[Sz\'{e}kelyhidi \cite{Sze85, Sze91}]\label{T1.3}
 Let $G$ be a topological abelian group which is generated by any neighbourhood of the 
zero, and let $X$ be a topological linear space, and $p:G\rightarrow X$ be a polynomial function. 
Then the following statements hold. 
\begin{enumerate}[(i)]
\item If $p:G\rightarrow X$ is continuous at a point, then it is 
continuous on $G$. 
\item Assume that $G$ is locally compact and $X$ is locally convex. 
If $p:G\rightarrow X$ is bounded on a measurable set of positive measure, 
then it is continuous. 
\item Suppose that $G$ is locally compact and $X$ is locally convex and locally bounded.  
If $p:G\rightarrow X$ is measurable on a measurable set of positive measure, 
then it is continuous. 
\end{enumerate}
\end{thm}

\section{Main results}

\subsection*{Preparatory statements}

In order to avoid superfluous repetitions, henceforth we will say that the
function in question is \emph{locally regular} on its domain, if at least 
one of the following statements are fulfilled. 
\begin{enumerate}[(i)]
 \item bounded on a measurable set of positive measure;
\item continuous at a point;
\item there exists a set of positive Lebesgue measure so that the restriction of the function in question
is measurable in the sense of Lebesgue. 
\end{enumerate}

Furthermore, a function will be called \emph{globally regular}, if instead of (ii), \\
(ii)' \; continuous on its domain \\
holds.

Firstly we prove a simple lemma. 

\begin{lem}\label{L1}
 Let $\alpha\in\mathbb{R}$ be arbitrarily fixed and let us assume that for the 
function $\phi\colon ]0, +\infty[\to\mathbb{R}$ the following statements are valid. 
\begin{enumerate}[(a)]
 \item the function $\phi$ is $\mathbb{Q}$-homogeneous of order $\alpha$, that is, 
\[
 \phi(rx)=r^{\alpha}\phi(x) 
\qquad 
\left(x\in ]0, +\infty[, r\in\mathbb{Q}\cap ]0, +\infty[\right). 
\]
\item the function $\phi$ is continuous at a point. 
\end{enumerate}
Then $\phi$ is continuous everywhere. 
\end{lem}
\begin{proof}
 Let us assume that the function $\phi$ is continuous at the point $x_{0}\in ]0, +\infty[$ and let 
$\tilde{x}\in ]0, +\infty[$ be arbitrary. 
Then, there exists a sequence of positive rational numbers $(r_{n})_{n\in\mathbb{N}}$ so that 
$\lim_{n\to\infty} r_{n}= \dfrac{\tilde{x}}{x_{0}}$. In this case 
the sequence $\left(\frac{1}{r_{n}}\right)_{n\in\mathbb{N}}$ is also a sequence of positive rational numbers and it 
converges to $\dfrac{x_{0}}{\tilde{x}}$. 
Due to property (a), 
\[
\frac{1}{r_{n}^{\alpha}}\phi(\tilde{x})=\phi\left(\frac{\tilde{x}}{r_{n}}\right) 
\qquad 
\left(n\in\mathbb{N}\right). 
\]
Taking the limit $n\to \infty$, the left hand side converges to 
$\left(\frac{x_{0}}{\tilde{x}}\right)^{\alpha}\phi(\tilde{x})$. Furthermore, 
the sequence $\left(\frac{\tilde{x}}{r_{n}}\right)_{n\in\mathbb{N}}$ converges to 
$x_{0}$, therefore the continuity of the function $\phi$ implies that 
the right hand side tends to $\phi(x_{0})$ as $n\to \infty$. 
This implies that 
\[
 \left(\frac{x_{0}}{\tilde{x}}\right)^{\alpha}\phi(\tilde{x})=\phi\left(\frac{x_{0}}{\tilde{x}}\tilde{x}\right)
\]
is fulfilled. 
Since $\tilde{x}\in ]0, +\infty[$ was arbitrary, we get that 
\[
 \phi(\lambda x)=\lambda^{\alpha}\phi(x) 
\qquad 
\left(\lambda, x\in ]0, +\infty[\right), 
\]
which obviously implies the (everywhere) continuity of the function $\phi$. 

\end{proof}

Furthermore, it is important pointing out the following fact. 
Fix $\alpha\in\mathbb{R}$ and let $\phi\colon ]0, +\infty[\to\mathbb{R}$ 
be a $\mathbb{Q}$-homogeneous function of order $\alpha$. 
Suppose that $\phi$ fulfils property (i) or (iii) on the set of positive Lebesgue measure 
$D$. 
Then, the $\mathbb{Q}$-homogeneity of the function $\phi$ implies that 
(i), respectively (iii) holds for the function $\phi$ on the set 
$r\cdot D$ for all $r\in\mathbb{Q}$. 

\begin{thm}\label{Thm2.1}
 Let $n, m\in\mathbb{Z}\setminus\left\{0\right\}$, $n\neq m$ so that 
$n=-m$ or $\mathrm{sign}(n)=\mathrm{sign}(m)$, let further 
$f, g\colon\mathbb{R}\to\mathbb{R}$ be additive functions. 
Define the function $\phi\colon \mathbb{R}\setminus\left\{0\right\}\to\mathbb{R}$ by the 
formula 
\[
 \phi(x)=f\left(x^{n}\right)-x^{n-m}g\left(x^{m}\right) 
\qquad 
\left(x\in\mathbb{R}\setminus\left\{0\right\}\right), 
\]
and assume that $\phi$ is locally regular. 
Then, the functions $F, G\colon\mathbb{R}\to\mathbb{R}$ defined by 
\[
 F(x)=f(x)-f(1)x \quad 
\text{and}
\quad 
G(x)=g(x)-g(1)x
\qquad 
\left(x\in\mathbb{R}\right)
\]
are derivations and 
\[
 nF(x)=mG(x)
\]
holds for arbitrary $x\in\mathbb{R}$. 
\end{thm}
\begin{proof}
 Concerning the values of $n$ and $m$ we have to distinguish three cases. 
At first, let us assume that $n, m>0$.
There is no loss of generality in assuming $n>m$. 
Define the function $\Phi$ on $\mathbb{R}^{n}$ by 
\begin{multline*}
 \Phi(x_{1}, \ldots, x_{n})=
f(x_{1}\cdots x_{n}) - 
\frac{1}{\binom{n}{m}} 
\sum_{\mathrm{card}(I)=m}
\left( \prod_{j \in \{1,2,\dots,n\} \setminus I} x_j \right)
g \left( \prod_{i \in I} x_i \right) 
\\
\left(x_{1}, \ldots, x_{n}\in\mathbb{R}\right), 
\end{multline*}
where the summation is considered for all subsets $I$ of cardinality $m$ of the index set 
$\left\{1, 2,\ldots, n\right\}$. 
Due to the additivity of the functions $f$ and $g$, the function 
$\Phi$ is a symmetric and $n$--additive function. Furthermore, its trace, that is, 
\[
 \Phi(x, \ldots, x)=\phi(x)=f\left(x^{n}\right)-x^{n-m}g\left(x^{m}\right) 
\qquad 
\left(x\in\mathbb{R}\right)
\]
is a polynomial function. On the other hand $\phi$ is a locally regular function.  
In view of Theorem \ref{T1.3}, this means that $\phi$ is a continuous polynomial function. 
Therefore, there exists $c\in\mathbb{R}$ such that 
\[
 \Phi(x_{1}, \ldots, x_{n})=cx_{1}\cdots x_{n} 
\quad
\left(x_{1}, \ldots, x_{n}\in\mathbb{R}\right), 
\]
that is, 
\[
 \phi(x)=cx^{n} 
\quad 
\left(x\in\mathbb{R}\right). 
\]
With the substitution $x=1$, we get $\phi(1)=c$. 
On the other hand, the definition of the function $\phi$ yields that 
$\phi(1)=f(1)-g(1)$. 
Thus, 
\[
 f(x^{n})-x^{n-m}g(x^{m})=\left[f(1)-g(1)\right]x^{n} 
\qquad 
\left(x\in\mathbb{R}\right). 
\]
Define the functions $F, G\colon \mathbb{R}\to\mathbb{R}$ by 
\[
 F(x)=f(x)-f(1)x 
\quad 
\text{and}
\quad 
G(x)=g(x)-g(1)x 
\qquad 
\left(x\in\mathbb{R}\right). 
\]
Then the above identity yields that
\[
 F\left(x^{n}\right)=x^{n-m}G(x^{m}) 
\qquad 
\left(x\in\mathbb{R}\right). 
\]
The statement of the theorem follows now from Theorem \ref{T1}. 

Secondly, let us assume that $n, m<0$. In this case we get that the function 
\[
 \phi(x)=f\left(x^{n}\right)-x^{n-m}g\left(x^{m}\right) 
\qquad 
\left(x\in\mathbb{R}\setminus\left\{0\right\}\right)
\]
is locally regular on its domain. Let $u\in\mathbb{R}\setminus\left\{0\right\}$, with the substitution $x=\frac{1}{u}$ this 
yields that 
\[
 \phi\left(\frac{1}{u}\right)=
f\left(u^{-n}\right)-u^{-(n-m)}g\left(u^{-m}\right) 
\qquad 
\left(u\in\mathbb{R}\setminus\left\{0\right\}\right). 
\]
Since $-n, -m>0$, the results of the previous case can be applied for the function 
\[
 \psi(u)=\phi\left(\frac{1}{u}\right) 
\qquad 
\left(u\in \mathbb{R}\setminus \left\{0\right\}\right), 
\]
which is, due to the local regularity of $\phi$, also locally regular. 

Finally, let us assume that $n=-m$. Without the loss of generality $m>0$ can be assumed. 
In this case 
\[
 \phi(x)=f\left(x^{-m}\right)-x^{-2m}g\left(x^{m}\right) 
\qquad 
\left(x\in\mathbb{R}\setminus\left\{0\right\}\right)
\]
is locally regular, or equivalently, the mapping 
\begin{equation}\label{Eq1}
\psi(x)= \phi\left(\sqrt[m]{x}\right)=f\left(\frac{1}{x}\right)-\frac{1}{x^{2}}g(x) 
\qquad 
\left(x>0\right)
\end{equation}
has the local regularity property. 
Let $u>0$ be arbitrary and  let us substitute $u(u+1)$ in place of $x$. Then 
\begin{multline*}
\psi(u(u+1))= \phi\left(\sqrt[m]{u(u+1)}\right)
\\
=f\left(\frac{1}{u(u+1)}\right)-\frac{1}{u^{2}(u+1)^{2}}g\left(u(u+1)\right)
\\
\left(u>0\right). 
\end{multline*}
Using the additivity of the function $f$, 
\begin{multline*}
\psi(u(u+1))=\phi\left(\sqrt[m]{u(u+1)}\right)
\\
=f\left(\frac{1}{u}\right)-f\left(\frac{1}{u+1}\right) -\frac{1}{u^{2}(u+1)^{2}}g\left(u(u+1)\right)
\\
\left(u>0\right). 
\end{multline*}
On the other hand, 
\[
\psi(u)=\phi\left(\sqrt[m]{u}\right)=f\left(\frac{1}{u}\right)-\frac{1}{u^{2}}g(u) 
\qquad 
\left(u>0\right)
\]
and 
\[
\psi(u+1)=\phi\left(\sqrt[m]{u+1}\right)=f\left(\frac{1}{u+1}\right)-\frac{1}{(u+1)^{2}}g(u+1) 
\qquad 
\left(u>0\right). 
\]
Therefore, 
\begin{multline*}
\psi(u(u+1))-\psi(u)+\psi(u+1)
\\
=
\phi\left(\sqrt[m]{u(u+1)}\right)-\phi(\sqrt[m]{u})+\phi(\sqrt[m]{u+1})
\\
=
f\left(\frac{1}{u}\right)-f\left(\frac{1}{u+1}\right) -\frac{1}{u^{2}(u+1)^{2}}g\left(u(u+1)\right)
\\
-f\left(\frac{1}{u}\right)+\frac{1}{u^{2}}g(u)+
f\left(\frac{1}{u+1}\right)-\frac{1}{(u+1)^{2}}g(u+1) 
\qquad 
\left(u>0\right)
\end{multline*}
Making use of the additivity of the function g, after rearrangement, we obtain that 
\[
 \chi(u)=2ug(u)-g\left(u^{2}\right) 
\qquad 
\left(u>0\right), 
\]
where 
\[
 \chi(u)=u^{2}(u+1)^{2}\left[\psi\left(u(u+1)\right)-\psi(u)+\psi(u+1)\right]+u^{2}g(1) 
\qquad 
(u>0). 
\]
By our assumptions, the function $\phi$  is locally regular on 
$\mathbb{R}\setminus\left\{0\right\}$ and due to the additivity of $f$ and $g$, 
it is $\mathbb{Q}$-homogeneous of order $n$. Thus, by Lemma \ref{L1}, 
$\phi$ is globally regular on $\mathbb{R}\setminus \left\{0\right\}$. 
This implies that $\psi$ is globally regular on $]0, +\infty[$, which means that 
$\chi$ is locally regular. 
Due to the results of the first case this yields that the function 
$G\colon\mathbb{R}\to\mathbb{R}$ defined by 
\[
 G(x)=g(x)-g(1)x 
\qquad 
\left(x\in\mathbb{R}\right)
\]
is a derivation. 
In view of \eqref{Eq1}, this implies that 
\[
\phi\left(\sqrt[m]{x}\right)=f\left(\frac{1}{x}\right)-\frac{1}{x^{2}}\left[G(x)+g(1)x\right] 
\qquad 
\left(x>0\right), 
\]
that is , 
\[
\phi\left(\sqrt[m]{x}\right)=f\left(\frac{1}{x}\right)+G\left(\frac{1}{x}\right)+g(1)\frac{1}{x} 
\qquad 
\left(x>0\right), 
\]
since $G$ is a derivation. 
Let $u>0$, with the substitution $x=\frac{1}{u}$ we get that 
\[
\psi(u)=\phi\left(\sqrt[m]{\frac{1}{u}}\right)=f(u)+G(u)+g(1)u
\qquad 
\left(u>0\right). 
\]
Let us observe that the right hand side of this identity is an additive function, being the sum of additive functions. 
Moreover, the left hand side is locally regular, due to the local regularity of $ \phi$. 
Thus $\psi$ is a regular additive function, which means that there exists $c\in\mathbb{R}$ so that 
\[
 f(u)+G(u)+g(1)u=cu 
\qquad 
\left(u\in\mathbb{R}\right). 
\]
With $u=1$, $c=f(1)+g(1)$ can be obtained, therefore, 
\[
 f(u)=\left[f(1)+g(1)\right]u-g(1)u+G(u) 
\qquad 
\left(u\in\mathbb{R}\right), 
\]
i.e., 
\[
 f(u)=-G(u)+f(1)u
\qquad 
\left(u\in\mathbb{R}\right). 
\]
This means that the function $F\colon \mathbb{R}\to\mathbb{R}$ defined by 
\[
 F(x)=f(x)-f(1)x 
\qquad 
\left(x\in\mathbb{R}\right)
\]
is a derivation and 
\[
 F(x)=-G(x) 
\qquad 
\left(x\in\mathbb{R}\right)
\]
holds. 
\end{proof}

\begin{lem}\label{L2.1}
 Let $\kappa\in\mathbb{R}$, $n, m\in\mathbb{Z}$, $n \neq m$ and assume that 
$f\colon\mathbb{R}\to\mathbb{R}$ is an additive function. 
Define the function $\phi\colon\mathbb{R}\setminus\left\{0\right\}\to\mathbb{R}$ by 
\[
 \phi(x)=f\left(x^{n}\right)-\kappa x^{n-m}f\left(x^{m}\right) 
\qquad 
\left(x\in\mathbb{R}\setminus \left\{0\right\}\right)
\]
and assume that $\phi$ is locally regular. 
Then, the function $F\colon\mathbb{R}\to\mathbb{R}$ defined by 
\[
 F(x)=f(x)-f(1)x \quad 
\left(x\in\mathbb{R}\setminus\left\{0\right\}\right)
\]
is a derivation so that for any $x\in\mathbb{R}$ 
\[
\left(n-\kappa m\right)F(x)=0. 
\]
\end{lem}
\begin{proof}
 In view of the previous theorem, it is enough to deal with the case $\mathrm{sign}(n)\neq\mathrm{sign}(m)$ and $n\neq -m$. 
Due to the definition of the function $\phi$ 
\[
 \phi\left(x^{n}\right)=f\left(x^{n^{2}}\right)-\kappa x^{n(n-m)}f\left(x^{nm}\right)  
\qquad 
\left(x\in\mathbb{R}\setminus\left\{0\right\}\right)
\]
and 
\[
 \phi\left(x^{m}\right)=f\left(x^{nm}\right)-\kappa x^{m(n-m)}f\left(x^{m^{2}}\right) 
\qquad 
\left(x\in\mathbb{R}\setminus\left\{0\right\}\right), 
\]
therefore 
\[
 \phi\left(x^{n}\right)+\kappa x^{n(n-m)}\phi\left(x^{m}\right)=
f\left(x^{n^{2}}\right)-\kappa^{2} x^{n^{2}-m^{2}}f\left(x^{m^{2}}\right) 
\qquad 
\left(x\in\mathbb{R}\setminus\left\{0\right\}\right). 
\]
By our assumptions, $\phi$ is a locally regular function on 
$\mathbb{R}\setminus\left\{0\right\}$. 
However, the additivity of $f$ implies that 
\[
 \phi(rx)=r^{n}\phi(x) 
\qquad 
\left(x\in\mathbb{R}\setminus\left\{0\right\}, r\in\mathbb{Q}\setminus \left\{0\right\}\right). 
\]
Using Lemma \ref{L1}, we get that $\phi$ is globally regular. 
Therefore, the function
\[
 \psi(x)=\phi\left(x^{n}\right)+\kappa x^{n(n-m)}\phi\left(x^{m}\right) 
\qquad 
\left(x\in\mathbb{R}\setminus\left\{0\right\}\right)
\]
is locally regular. 
Since $n^{2}, m^{2}>0$ and $n^{2}\neq m^{2}$, the results of the previous theorem can be applied 
(with the choice $\psi(x)=\phi\left(x^{n}\right)+\kappa x^{n(n-m)}\phi\left(x^{m}\right)$ and $g(x)=\kappa^{2}f(x)$) 
to obtain that 
\[
 f(x)=F(x)+f(1)x 
\qquad 
\left(x\in\mathbb{R}\right), 
\]
where $F\colon \mathbb{R}\to\mathbb{R}$ is a derivation and 
\[
 nF(x)=m\kappa F(x)
\]
is also fulfilled for all $x\in\mathbb{R}$.

\end{proof}

From this lemma, the following corollary can be concluded immediately. 

\begin{cor}\label{C2.1}
 Let $r\in\mathbb{Q}\setminus\left\{0, 1\right\}$ be arbitrarily fixed and 
$f\colon\mathbb{R}\to\mathbb{R}$ be an additive function and define the function
 by
\[
 \phi(x)=f\left(x^{r}\right)-rx^{r-1}f\left(x\right)
\qquad 
\left(x\in\mathbb{R},\, x>0\right), 
\]
and assume that $\phi$ is locally regular. 
Then, the function $F\colon\mathbb{R}\to\mathbb{R}$ defined by 
\[
 F(x)=f(x)-f(1)x \quad 
\left(x\in\mathbb{R}\right)
\]
is a derivation. 
\end{cor}

\subsection*{Derivations along rational functions}

In view of the results of the previous subsection, we are able to prove the following

\begin{thm}\label{T2.2}
 Let $n\in\mathbb{Z}\setminus\left\{0\right\}$ and 
$\left(\begin{array}{cc}
a&b\\
c&d
\end{array}
\right)\in\mathbf{GL}_{2}(\mathbb{Q})$ be such that 
\begin{enumerate}[--]
 \item if $c=0$, then $n\neq 1$;
\item if $d=0$, then $n\neq -1$. 
\end{enumerate}
Let further $f, g\colon\mathbb{R}\to\mathbb{R}$ be additive functions and define the function 
$\phi$ by 
\[
 \phi(x)=f\left(\frac{ax^{n}+b}{cx^{n}+d}\right)-\frac{x^{n-1}g(x)}{\left(cx^{n}+d\right)^{2}} 
\qquad 
\left(x\in\mathbb{R}, \,
cx^{n}+d\neq 0\right). 
\]
Let us assume $\phi$ to be globally regular. 
Then, the functions $F, G\colon\mathbb{R}\to\mathbb{R}$ defined by 
\[
 F(x)=f(x)-f(1)x \quad \text{and} 
\quad 
G(x)=g(x)-g(1)x
\quad
\left(x\in\mathbb{R}\right)
\]
are derivations. 
\end{thm}

\begin{proof}
 Firstly, let us suppose that  $c=0$. This means that the function 
\[
 \phi(x)=f\left(\frac{a}{d}x^{n}+\frac{b}{d}\right)-\frac{1}{d^{2}}x^{n-1}g(x)
\qquad 
\left(x\in\mathbb{R}\setminus\left\{0\right\}\right)
\]
is globally regular.  
In this case, the statement immediately follows from Theorem \ref{Thm2.1}. 

Similarly, if $d=0$, then 
\[
 \phi(x)=f\left(\frac{a}{c}+\frac{b}{c}x^{-n}\right)-x^{-n-1}g(x) 
\qquad 
\left(x\in\mathbb{R}\setminus\left\{0\right\}\right)
\]
is globally regular. Therefore, due to Theorem \ref{Thm2.1}, we obtain that the functions 
\[
 F(x)=f(x)-f(1)x \quad 
\text{and}
\quad 
G(x)=g(x)-g(1)x 
\qquad 
\left(x\in\mathbb{R}\right)
\]
are derivations. 

Thus, henceforth $cd\neq 0$ can be assumed. Furthermore, due to the $\mathbb{Q}$-homogeneity of the functions 
$f$ and $g$, $c=1$ can be supposed. That is, 
\[
 \phi(x)=f\left(\frac{ax^{n}+b}{x^{n}+d}\right)-\frac{x^{n-1}g(x)}{\left(x^{n}+d\right)^{2}} 
\qquad 
\left(x\in\mathbb{R}, \,
x^{n}+d\neq 0\right). 
\]
Since the function $f$ is additive, 
\[
 f\left(\frac{ax^{n}+b}{x^{n}+d}\right)=f(a)-f\left(\frac{D}{x^{n}+d}\right), 
\qquad 
\left(x\in\mathbb{R}, \,
x^{n}+d\neq 0\right), 
\]
therefore, 
\begin{equation}\label{Eq2}
 \phi(x)=f(a)-f\left(\frac{D}{x^{n}+d}\right)-\frac{x^{n-1}g(x)}{(x^{n}+d)^{2}}
\qquad 
\left(x\in\mathbb{R}, \,
x^{n}+d\neq 0\right), 
\end{equation}
where $D=\det\left(\begin{array}{cc}
                    a&b\\
1&d
                   \end{array}
\right)$. 
Let us observe that 
\[
 \frac{D}{x^{n}+d}=\frac{D}{d}-\frac{D}{\left(\sqrt[n]{d^{2}}\frac{1}{x}\right)^{n}+d}
\]
holds for all $x\in\mathbb{R}, \, x\neq 0, x^{n}+d\neq 0.$
Using this identity, we receive
\begin{multline}\label{Eq3}
\phi(x)=
f(a)-f\left(\frac{D}{d}\right)+f\left(\frac{D}{\left(\sqrt[n]{d^{2}}\frac{1}{x}\right)^{n}+d}\right)
-\frac{x^{n-1}g(x)}{\left(x^{n}+d\right)^{2}}
\\
\left(x\in\mathbb{R}, \, x\neq 0, x^{n}+d\neq 0\right), 
\end{multline}
where the additivity of the function $f$ was also used. 
Let us replace $x$ by $\sqrt[n]{d^{2}}\dfrac{1}{x}$ in \eqref{Eq2} to acquire 
\begin{multline*}
 \phi\left(\sqrt[n]{d^{2}}\frac{1}{x}\right)=
f(a)-f\left(\frac{D}{\left(\sqrt[n]{d^{2}}\frac{1}{x}\right)^{n}+d}\right)-
\frac{\left(\sqrt[n]{d^{2}}\frac{1}{x}\right)^{n-1}g\left(\sqrt[n]{d^{2}}\frac{1}{x}\right)}
{\left(\left(\sqrt[n]{d^{2}}\frac{1}{x}\right)^{n}+d\right)^{2}}
\\
\left(x\in\mathbb{R}, \, x\neq 0, x^{n}+d\neq 0\right). 
\end{multline*}
Since 
\[
 \left(\sqrt[n]{d^{2}}\frac{1}{x}\right)^{n}+d= d\frac{1}{x^{n}}\left(x^{n}+d\right), 
\]
the above identity yields that 
\begin{multline*}
\phi\left(\sqrt[n]{d^{2}}\frac{1}{x}\right)=
f(a)-f\left(\frac{D}{\left(\sqrt[n]{d^{2}}\frac{1}{x}\right)^{n}+d}\right)-
\frac{\left(\sqrt[n]{d^{2}}\frac{1}{x}\right)^{n-1}g\left(\sqrt[n]{d^{2}}\frac{1}{x}\right)}
{\left(d\frac{1}{x^{n}}\right)^{2}\left(\left(x^{n}+d\right)\right)^{2}} 
\\
\left(x\in\mathbb{R}, \, x\neq 0, x^{n}+d\neq 0\right).
\end{multline*}
After some rearrangement, we arrive at 
\begin{multline}\label{Eq4}
 \phi\left(\sqrt[n]{d^{2}}\frac{1}{x}\right)=
f(a)-f\left(\frac{D}{\left(\sqrt[n]{d^{2}}\frac{1}{x}\right)^{n}+d}\right)-
\frac{x^{n-1}\frac{1}{\sqrt[n]{d^{2}}}x^{2}g\left(\sqrt[n]{d^{2}}\frac{1}{x}\right)}{\left(x^{n}+d\right)^{2}}
\\
\left(x\in\mathbb{R}, \, x\neq 0, x^{n}+d\neq 0\right).
\end{multline}
In case we add \eqref{Eq3} and \eqref{Eq4} together, 
\begin{multline*}
 \phi(x)+\phi\left(\sqrt[n]{d^{2}}\frac{1}{x}\right)=
f\left(2a-\frac{D}{d}\right)-
\frac{x^{n-1}}{\left(x^{n}+d\right)^{2}} 
\left[
g(x)+x^{2}\frac{1}{\sqrt[n]{d^{2}}}g\left(\sqrt[n]{d^{2}}\frac{1}{x}\right)
\right]
\\
\left(x\in\mathbb{R}, \, x\neq 0, x^{n}+d\neq 0\right).
\end{multline*}
Let us define the functions
\[
 h(x)=\frac{1}{\sqrt[n]{d^{2}}}g\left(\sqrt[n]{d^{2}}x\right)
\qquad 
\left(x\in\mathbb{R}\right)
\]
and 
\begin{multline*}
 \psi(x)=
-\frac{1}{x^{2}}\frac{\left(x^{n}+d\right)^{2}}{x^{n-1}}
\left[
\phi(x)+\phi\left(\sqrt[n]{d^{2}}\frac{1}{x}\right)-f\left(2a-\frac{D}{d}\right)
\right]
\\
\left(x\in\mathbb{R}, \, x\neq 0, x^{n}+d\neq 0\right).
\end{multline*}
In this case
\[
 \psi(x)=h\left(\frac{1}{x}\right)+\frac{1}{x^{2}}g(x)
\quad 
\left(x\in\mathbb{R}, \, x\neq 0, x^{n}+d\neq 0\right)
\]
holds. 
By our assumptions $\phi$ is a globally regular mapping, therefore the function 
$\psi$ has the local regularity property. 
Due to Theorem \ref{Thm2.1}, this gives that the functions 
$F, G\colon\mathbb{R}\to\mathbb{R}$ defined by 
\[
 F(x)=f(x)-f(1)x
\quad 
\text{and}
\quad 
G(x)=g(x)-g(1)x 
\quad 
\left(x\in\mathbb{R}\right)
\]
are derivations.
\end{proof}

\subsection*{A characterization of linearity}

Finally, in the last part of this paper we will present a characterization of linearity. 
Just as in the proof of Theorem \ref{T2.2}, Theorem \ref{Thm2.1} will again play an important role. 

\begin{thm}
 Let $n\in\mathbb{N}, \, n\neq 1$ and $f\colon\mathbb{R}\to\mathbb{R}$ be an additive function. 
Define $\phi$ on $\mathbb{R}$ by
\[
 \phi(x)=f\left(x^{n}\right)-f(x)^{n} 
\qquad 
\left(x\in\mathbb{R}\right). 
\]
Let us assume that $\phi$ is locally regular. 
Then the function $f$ is linear, that is, 
\[
 f(x)=f(1)x
\]
holds for all $x\in\mathbb{R}$. 
\end{thm}
\begin{proof}
 Let us define the function $\Phi\colon\mathbb{R}^{n}\to\mathbb{R}$ by 
\[
 \Phi(x_{1}, \ldots, x_{n})=f\left(x_{1}\cdots x_{n}\right)-f(x_{1})\cdots f(x_{n}) 
\qquad 
\left(x_{1}, \ldots, x_{n}\in\mathbb{R}\right). 
\]
Due to the additivity of $f$, the function $\Phi$ is a symmetric, $n$-additive function. Furthermore, 
\[
 \Phi(x, \ldots, x)=\phi(x)=f\left(x^{n}\right)-f(x)^{n}
\qquad 
\left(x\in\mathbb{R}\right). 
\]
From the local regularity of the function $\phi$  we immediately deduce that 
$\phi$ is a continuous polynomial function. 
Consequently, 
\begin{equation}\label{Eq5}
 \Phi(x_{1}, \ldots, x_{n})=cx_{1}\cdots x_{n} 
\qquad 
\left(x_{1}, \ldots, x_{n}\in\mathbb{R}\right)
\end{equation}
holds with a certain $c\in\mathbb{R}$. Due to the definition of the 
function $\phi$, we have $\phi(1)=f(1)-f(1)^{n}$. On the other hand 
\[
 \phi(1)=\Phi(1, \ldots, 1)=c. 
\]
Hence $c=f(1)-f(1)^{n}$. 
Let $u\in\mathbb{R}$, with the substitution 
\[
 x_{1}=u, \qquad x_{i}=1 \, \text{for $i\geq 2$}, 
\]
equation \eqref{Eq5} yields that 
\[
 f(u)-f(u)f(1)^{n-1}=\left(f(1)-f(1)^{n}\right)u  
\qquad 
\left(u\in\mathbb{R}\right). 
\]
In case $f(1)^{n-1}\neq 1$, this furnishes 
\[
 f(u)=f(1)u 
\qquad 
\left(u\in\mathbb{R}\right). 
\]
If $f(1)^{n-1}=1$, then 
\[
 c=f(1)-f(1)^{n}=f(1)\left[1-f(1)^{n-1}\right]=0. 
\]
Therefore equation \eqref{Eq5} with the substitutions 
\[
 x_{1}=u, \quad x_{2}=v, \quad \text{and} \quad x_{i}=1 \quad \text{for $i\geq 3$} 
\qquad 
\left(u, v\in\mathbb{R}\right)
\]
yields that 
\[
 f(uv)=f(u)f(v)f(1)^{n-2} 
\qquad 
\left(u, v\in\mathbb{R}\right), 
\]
that is, $f(1)^{n-2}\cdot f$ is a non identically zero real homomorphism. 
In view of Kuczma \cite[Theorem 14.4.1.]{Kuc09} this implies that 
\[
 f(1)^{n-2}f(u)=u 
\qquad 
\left(u\in\mathbb{R}\right). 
\]
Since $f(1)^{n-2}\cdot f(1)=f(1)^{n-1}=1$, 
\[
 \frac{f(u)}{f(1)}=u
\]
holds for all $u\in\mathbb{R}$, that is, $f$ is a linear function, indeed. 
\end{proof}

\subsection*{Acknowledgements}
This paper is dedicated to the 60\textsuperscript{\textrm{t}h} birthday of 
Professor L\'{a}szl\'{o} Sz\'{e}kelyhidi (University of Debrecen, Hungary). 
Furthermore, the author wishes to express her gratitude to the two anonymous referees 
for their work.

\end{document}